\documentclass[reqno,11pt,centertags,draft]{amsart}
\usepackage{amsmath,amsthm,amscd,amssymb,latexsym,upref}
\date{\today}
\newcommand{\bT}{{\mathbb{T}}}

\renewcommand{\a}{\alpha}
\newcommand{\be}{\beta}

\newcommand{\La}{\Lambda}

\allowdisplaybreaks
\numberwithin{equation}{section}
\newtheorem{theorem}{Theorem}[section]

\newtheorem{lemma}[theorem]{Lemma}

\theoremstyle{definition}

\newtheorem{example}[theorem]{Example}

\begin{document}

\title[]{On Baxter's difference systems}
\author[]{J. S. Geronimo, L. Golinskii}

%\thanks{$^*$The work was partially supported by the Austrian Science
%Found FWF, project number: P16390--N04 }

\address{Georgia Institute of Technology, USA}
\email{geronimo@math.gatech.edu}

\address{Institute for Low Temperature Physics and Engineering, Ukraine}
\email{leonid.golinskii@gmail.com}

\date{\today}

\begin{abstract}
We study the asymptotics of solutions of a difference system
introduced by Baxter by using the general method for the asymptotic
representation of such solutions due to Benzaid and Lutz. Some
results of Tauberian type are obtained in the case when the spectral
parameter belongs to the unit circle.
\end{abstract}

\maketitle

\section{Introduction}

In this paper we study the asymptotics of certain solutions to a
system of difference equations introducded by Baxter \cite{Baxa, Baxb, Baxc}.
Baxter himself was interested in generalizing some results that had
been obtained by Szego, Geronimus, and Verblunsky on the difference
system satisfied by polynomials orthogonal on the unit circle. This
system has been of much interest since it appears in the
study by Miller et al \cite{M} and Gesztesy et al \cite{Gesa, Gesb} of solutions of the Ablowitz-Ladik
equations and by Geronimo et.al \cite{G} who generalized the results
of Geronimo and Johnson \cite{GJ} (see also \cite{simA, simB}).

\section{The Baxter equations}

Let $\mu$ be a complex Borel measure on the unit circle $\bT$ of
bounded total variation. We assume throughout that the $n$-th
Toeplitz determinant associated with $\mu$
$$
D_n(\mu)=\det\|\mu_{i-j}\|_{i,j=0}^{n-1}\not=0,\qquad
\mu_k=\int_{\bT}\zeta^{-k}d\mu,
$$
for all $n=1,2,\ldots$. Construct polynomials $\hat\phi_n$ and
$\hat\psi_n$ of degree $n$ in $z$ and ${1\over z}$, respectively, by
$$
\hat\phi_n(z)=
\begin{vmatrix}\mu_0 & \mu_{-1} & \ldots &\mu_{-n}\\
                   \mu_1 & \mu_0 & \ldots & \mu_{-n+1}\\
                   \ldots & \ldots & \ldots & \ldots\\
                   \mu_{n-1} & \mu_{n-2} & \ldots & \mu_{-1}\\
                   1 & z & \ldots & z^n
                   \end{vmatrix} = D_n(\mu)z^n+\ldots
$$
and
$$
\hat\psi_n(z)=
\begin{vmatrix}\mu_0 & \mu_{1} & \ldots & \mu_{n}\\
                   \mu_{-1} & \mu_0 & \ldots & \mu_{n-1}\\
                   \ldots & \ldots & \ldots & \ldots\\
                   \mu_{-n+1} & \mu_{-n+2} & \ldots & \mu_{1}\\
                   1 & z^{-1} & \ldots & z^{-n}
                   \end{vmatrix}=D_n(\mu)z^{-n}+\ldots
$$
such that,
$$
\int_{\bT}\hat\phi_n(\zeta)\zeta^{-k}d\mu=
\int_{\bT}\hat\psi_n(\zeta)\zeta^{k}d\mu = D_{n+1}(\mu)\delta_{nk},
\quad k=0,1,\ldots,n. $$
 Pick $\ell_n$ from $\ell_n^2=(D_n D_{n+1})^{-1}$ ($\ell_n$ is
 determined up to a $\pm$ sign) and put
\begin{equation*}
\begin{split}
\phi_n(z)&:=\ell_n\hat\phi_n(z)=\kappa_nz^n+\ldots, \\
\psi_n(z)&:=\ell_n\hat\psi_n(z)=\kappa_nz^{-n}+\ldots, \quad
\kappa_n=\ell_nD_n(\mu)\ne0. \end{split}
\end{equation*}
 So
$$
\int_{\bT}\hat\phi_n(\zeta)\zeta^{-k}d\mu=
\int_{\bT}\hat\psi_n(\zeta)\zeta^{k}d\mu = {\delta_{nk}\over
\kappa_n}, \quad k=0,1,\ldots,n, $$ and $\{\phi_n,\psi_n\}$ form a
bi-orthogonal sequence
$$ \int_{\bT} \phi_n(\zeta)\psi_m(\zeta)d\mu =
\delta_{n,m}.
$$
The polynomials will be unique once a determination is made for
their leading coefficients.

The {\it Baxter parameters} are introduced in the following way
$$
\beta_{n}:= -\kappa_{n-1}\int_{\bT}\phi_{n-1}(\zeta)\zeta d\mu, \ \
\alpha_{n}:=-\kappa_{n-1}\int_{\bT}\psi_{n-1}(\zeta)\zeta^{-1} d\mu,
\quad n=1,2,\ldots.
$$
It is easy to see that
\begin{equation*}
\begin{split}
\beta_{n} &=-\ell_{n-1}^2D_{n-1}(\mu)
\begin{vmatrix} \mu_0 & \mu_{-1} &\ldots &\mu_{-n+1}\\
                \mu_1 & \mu_0 & \ldots & \mu_{-n+2}\\
                \ldots & \ldots & \ldots & \ldots\\
                \mu_{n-2} & \mu_{n-3} & \ldots & \mu_{-1}\\
                \mu_{-1} & \mu_{-2} & \ldots & \mu_{-n}
                \end{vmatrix} \\
&={(-1)^{n}\over D_{n}(\mu)}
\begin{vmatrix} \mu_{-1} & \mu_{-2} &\ldots &\mu_{-n}\\
                \mu_0 & \mu_{-1} & \ldots & \mu_{-n+1}\\
                \ldots & \ldots & \ldots & \ldots\\
                \mu_{n-2} & \mu_{n-3} & \ldots & \mu_{-2}\\
                \mu_{n-2} & \mu_{n-3} & \ldots & \mu_{-1}
                \end{vmatrix}
=(-1)^{n} {D_{n}(\zeta\mu)\over D_{n}(\mu)}.
\end{split}\end{equation*}
Similarly,
$$
\alpha_{n}=(-1)^{n} {D_{n}(\zeta^{-1}\mu)\over D_{n}(\mu)}.
$$
Next, we have
$$
1-\alpha_n\beta_n={D_n^2(\mu)-D_n(\zeta\mu)D_n(\zeta^{-1}\mu)\over
D_n^2(\mu)}.
$$
On the other hand, the Silvester identity applied to $D_{n+1}(\mu)$
gives
$$
D_{n+1}(\mu)D_{n-1}(\mu)=D_n^2(\mu)-D_n(\zeta\mu)D_n(\zeta^{-1}\mu),
$$
so
$$
1-\alpha_n\beta_n={D_{n+1}(\mu)D_{n-1}(\mu)\over D_n^2(\mu)}.
$$
Since $\kappa_n^2=\ell_n^2D_n^2(\mu)=D_n(\mu)/D_{n+1}(\mu)$, then
\begin{equation}\label{baxterpar}
1-\alpha_n\beta_n={\kappa_{n-1}^2\over \kappa_n^2}\,,\qquad
\kappa_n^{-2}=\prod_{j=1}^n (1-\alpha_j\beta_j).
\end{equation}

We define two sequences $\{u_n\}$ and $\{v_n\}$ of polynomials of
degree $n$ in $z$ and ${1\over z}$, respectively, by
\begin{equation}\label{un}
u_n(z):=\kappa_nz^n\psi_n(z)=\frac1{D_{n+1}(\mu)}
\begin{vmatrix} \mu_0 & \mu_{1} & \ldots & \mu_{n} \\
                \mu_{-1} & \mu_0 & \ldots & \mu_{n-1}\\
                \ldots & \ldots & \ldots & \ldots\\
                \mu_{-n+1} & \mu_{-n+2} & \ldots & \mu_{1}\\
                z^n & z^{n-1} & \ldots & 1
                \end{vmatrix},\end{equation}

\begin{equation}\label{vn}
v_n(z):=\kappa_nz^{-n}\phi_n(z)=\frac1{D_{n+1}(\mu)}
\begin{vmatrix} \mu_0 & \mu_{-1} & \ldots &\mu_{-n}\\
                \mu_1 & \mu_0 & \ldots & \mu_{-n+1}\\
                \ldots & \ldots & \ldots & \ldots\\
                \mu_{n-1} & \mu_{n-2} & \ldots & \mu_{-1}\\
                z^{-n} & z^{-n+1} & \ldots & 1\end{vmatrix}.
\end{equation}
It is easy to check that $\beta_n u_n$ and
$\kappa_n\phi_n-\kappa_{n-1}z\phi_{n-1}$, both polynomials in $z$ of
degree $\le n$, have the same Fourier coefficients with respect to
$\mu$. The same goes for $\alpha_n v_n$ and
$\kappa_n\psi_n-\kappa_{n-1}z^{-1}\psi_{n-1}$. Hence
\begin{equation*}
\begin{split}
\beta_nu_n(z) &=\kappa_n\phi_n(z)-\kappa_{n-1}z\phi_{n-1}(z), \\
\alpha_n v_n(z) &=\kappa_n\psi_n(z)-\kappa_{n-1}z^{-1}\psi_{n-1}(z).
\end{split}
\end{equation*}
In the matrix form
\begin{equation*}
\begin{split}
\Phi_{n-1}(z) &=\frac{\kappa_{n}}{\kappa_{n-1}}
\begin{bmatrix} z^{-1}&-z^{-1}\beta_{n}\\
-\alpha_{n}&1 \end{bmatrix}\,\Phi_n(z), \\
\Phi_n(z) &=\begin{bmatrix} \phi_n(z)\\
z^n\psi_n(z)\end{bmatrix}=\begin{bmatrix}\phi_n(z)\\
\kappa_n^{-1}u_n(z)\end{bmatrix}.
\end{split}
\end{equation*}
Taking inverse matrices and shifting indices we have by
\eqref{baxterpar}
\begin{equation}\label{baxtereqone}
\Phi_{n+1}(z) =\frac{\kappa_{n+1}}{\kappa_{n}}
\begin{bmatrix} z & \beta_{n+1}\\
\alpha_{n+1}z&1 \end{bmatrix}\,\Phi_n(z).
\end{equation}
Likewise we obtain
\begin{equation}\label{baxtereqtwo}
\Psi_{n+1}(z) =\frac{\kappa_{n+1}}{\kappa_{n}}
\begin{bmatrix} z^{-1} & \a_{n+1}\\
\be_{n+1}z^{-1} & 1 \end{bmatrix}\,\Psi_n(z), \quad
\Psi_n(z) =\begin{bmatrix} \psi_n(z)\\
\kappa_n^{-1}v_n(z)\end{bmatrix}.
\end{equation}
Inspection of the above equations shows that
\begin{equation}\label{phipsi}
\Psi_n(z;\alpha,\beta)=\Phi_n(z^{-1};\beta,\alpha)
\end{equation}
(since $\phi_0=\psi_0=1$).

In the case when $\beta_n=\bar\alpha_n,\ |\alpha_n|<1 $ for all $n$,
the equations of Baxter reduce to those satisfied by polynomial
orthogonal with respect to some positive measure supported on the
unit circle.

\section{The Benzaid and Lutz method}

As mentioned in the introduction, Benzaid and Lutz \cite{BL}
developed methods for the asymptotic representation for solutions of
difference equations. We will be interested in the case when the
coefficients in equation \eqref{baxtereqone} satisfy
\begin{equation}\label{ltwo}
\sum_{n=1}^{\infty}|\alpha_n|^2 < \infty,\qquad
\sum_{n=1}^{\infty}|\beta_n|^2<\infty,\qquad {\rm and}\qquad
\alpha_n\beta_n\ne 1\quad \forall n\geq1. \end{equation}
 If in equation \eqref{baxtereqone} we switch to the monic polynomials
$\hat\Phi_n=\kappa_n^{-1}\Phi_n$, it becomes
\begin{equation}\label{baxtermonic}
\begin{split}
\hat\Phi_{n+1} (z) &=\begin{bmatrix} z & \beta_{n+1}\\ \alpha_{n+1}
z & 1\end{bmatrix}\,\hat\Phi_n (z) =\left[\Lambda(z) +
U(z,n)\right]\,\hat\Phi_n(z), \\ U(z,n)&= \begin{bmatrix} 0 &
\beta_{n+1}\\ \alpha_{n+1}z & 0\end{bmatrix}, \qquad
\Lambda(z)=\begin{bmatrix} z & 0\\ 0 & 1\end{bmatrix}.
\end{split}
\end{equation}
The unperturbed system is
$\hat\Phi^0_{n+1}(z)=\Lambda(z)\hat\Phi^0_n(z)$, and following
Benzaid and Lutz we choose a $Q(z,n)$ such that ${\rm diag}
(Q(z,n))=0$ satisfying the equation
\begin{equation}\label{benzlutz}
U(n)+\Lambda(z) Q(z,n)-Q(z,n+1)\Lambda(z) =0. \end{equation} It is
easy to check that $Q$ in \eqref{benzlutz} can be taken as
$$ Q(z,n)=\begin{bmatrix} 0&q_{12}(z,n)\\ q_{21}(z,n)&0\end{bmatrix} $$
with
\begin{equation}\label{qij}
\begin{split}
q_{12}(z,n)&=\sum^n_{k=1} \beta_k z^{n-k}= \sum^{n-1}_{j=0}
\beta_{n-j} z^j, \\
q_{21}(z,n)&=-\sum^\infty_{k=n+1}\alpha_k
z^{k-n}=-\sum^\infty_{j=1}\alpha_{n+j} z^j.
\end{split}
\end{equation}
We could have defined a new variable $X$ from
$\hat\Phi_n(z)=(I+Q(z,n))X(z,n)$, and so
$$ X(z,n+1)=(I+Q(z,n+1))^{-1}(\Lambda(z)+U(z,n))(I+Q(z,n))X(z,n). $$
However in order not to introduce a denominator when computing
inverses $(I+Q)^{-1}$ we will carry out the Benzaid and Lutz
procedure in two steps using nilpotent matrices
$$Q_1(z,n)=\begin{bmatrix} 0 & q_{12}(z,n)\\ 0 & 0\end{bmatrix},\qquad
Q_2(z,n)= \begin{bmatrix} 0 & 0\\ q_{21}(z,n) & 0\end{bmatrix},
$$
satisfying the equations
\begin{equation}\label{benzlutzone}
U_j(n)+\Lambda(z) Q_j(z,n)-Q_j(z,n+1)\Lambda(z) =0, \qquad j=1,2
\end{equation} with
$$ U_1(z,n)=\begin{bmatrix} 0&0 \\ \alpha_{n+1}z &0\end{bmatrix}, \qquad U_2(z,n)=
\begin{bmatrix} 0&\beta_{n+1}\\ 0&0 \end{bmatrix},
$$
respectively.

Some useful properties of the above functions are

\begin{lemma}\label{bzlu}
With the conditions given by $\eqref{ltwo}$:
\begin{enumerate}
\item For $|z|<1$  $q_{ij}(z,n)\to 0$ as $n\to\infty$, both are
finite for almost every $z\in{\bT}$;
\item For almost every
$z\in\bT$ $q_{21}(z,n)\to 0$ as $n\to\infty$;
\item For $|z|<1$,
$q_{ij}(z,n)\in\ell_2$ and belong to $H^2$ as functions of $z$ for
all $n\geq 1$.
\end{enumerate}
\end{lemma}
\begin{proof}  For $|z|<1$, it was noted by Benzaid and Lutz that (1) and (3)
follow since $q_{ij}$ is the convolution of an $l^2$ sequence with
an $l^1$ sequence. For $|\zeta|=1$ the results follow from the
celebrated Carleson's theorem.
\end{proof}

For $|z|<1$ define $Y(z,n)$ from
\begin{equation}\label{newvar}
\hat\Phi_n (z)=(I+Q_1(z,n))(I+Q_2(z,n))Y(z,n).\end{equation}
 Then because of \eqref{baxtermonic} and $(I+Q_j)^{-1}=I-Q_j$, we have
\begin{equation*}
\begin{split}
Y(z, n+1)= &(I-Q_2(z,n+1))(I-Q_1(z,n+1))(\La+U(n)) (I+Q_1(z,n))
\\ &(I+Q_2(z,n))Y(z,n). \end{split}
\end{equation*}
After multiplying out the brackets the ``first order terms'' vanish
(that is the key idea of Benzaid and Lutz), and we come to
$$
Y(z,n+1)=(\Lambda(z)+V(z,n))Y(z,n), \qquad V(z,n)=\begin{bmatrix}
v_{11}(z,n)& v_{12}(z,n)\\ v_{21}(z,n)& v_{22}(z,n) \end{bmatrix},
$$
with
\begin{equation}\label{vij}
\begin{split}
v_{11}(z,n) &=\beta_{n+1}q_{21}(z,n)(1+\rho(z,n+1)),\\
v_{12}(z,n)
&=\beta_{n+1}(\rho(z,n+1)+\rho(z,n)+\rho(z,n)\rho(z,n+1)),\\
v_{21}(z,n) &=-\beta_{n+1}q_{21}(z,n)q_{21}(z,n+1),\\
v_{22}(z,n) &=-\beta_{n+1}q_{21}(z,n+1)(1+\rho(z,n)),\end{split}
\end{equation}
where $\rho(z,n)=q_{12}(z,n)q_{21}(z,n)$. By Lemma \ref{bzlu}
$V(z,n)\in\ell_1$ for each $|z|<1$.

Let $T_r=\{|z|=r\}$ be a circle of a radius $0<r<1$. For $z\in T_r$
we define \ $W(z,n)=\La^{-1}V(z,n)\in\ell^1$, then
$$ Y(z,n+1)=\La(z)(I+W(z,n))Y(z,n). $$
Consequently, for $z\in T_r$
\begin{equation}\label{firstbound}
\begin{split}
\|Y(z,n+1)\| &\le\left\| \La(z)\right\|\|I+W(z,n)\| \|Y(z,n)\|\\
&\le \|I+W(z,n)\| \| Y(z,n)\|\le (1+\|W(z,n)\|)\|Y(z,n)\|\cr &\le
\|Y(z,0)\|\exp\sum^\infty_{n=0} \|W(z,n)\|< c_1,
\end{split}
\end{equation}
 where
$$ c_1=\sup_{|z|=r}\|Y(z,0)\|\exp\sum^\infty_{0} \|W(z,n)\|. $$
It now follows by induction that
\begin{equation}\label{convone}
\left\|Y(z,n)-\begin{bmatrix} z^{n-m}&0\\
0&1\end{bmatrix}\,Y(z,m)\right\|<c_1\sum_{k=m}^n\|W(z,k)\|.
\end{equation}
Let us write the latter inequality for the vector function
$Y(z,n)=\begin{bmatrix} y_{1}(z,n)\\ y_{2}(z,n)\end{bmatrix} $
coordinatewise
\begin{equation}\label{coordin}
|y_{1}(z,n)-z^{n-m}y_{1}(z,m)|+|y_{2}(z,n)-y_{2}(z,m)|\le
c_1\sup_{|z|=r}\sum_{k=m}^n\|W(z,k)\|\end{equation} for $|z|=r$, and
hence by the Maximum Modulus Theorem for $|z|\le r$. So
\begin{equation}\label{ualbe}
\lim_{n\to\infty}y_{1}(z,n)=0, \quad
\lim_{n\to\infty}y_{2}(z,n)=u(z;\alpha,\beta),\end{equation}
 $u$ is an analytic function, uniformly on compact subsets of the unit disk.
From the relation between $Y(n)$ and $\hat\Phi_n$ \eqref{newvar}
\begin{equation}\label{psiytwo}
\hat\Phi_n(z)= \begin{bmatrix} 1&q_{12}(z,n)\\
q_{21}(z,n)&1+\rho(z,n)\end{bmatrix} \begin{bmatrix} y_{1}(z,n)\\
y_{2}(z,n)\end{bmatrix}\end{equation}
and Lemma \ref{bzlu} we see
that the above implies that $\{\hat\Phi_n(z)\}$ converge uniformly
on compact subsets of $|z|<1$ to
\begin{equation}\label{hatphi}
\hat\Phi(z)=\begin{bmatrix} 0\\
u(z;\alpha,\beta)\end{bmatrix}.\end{equation}

Since by \eqref{baxterpar} the leading coefficients $\kappa_n$
converge, the same conclusion holds for $\Phi_n$.

Now, let $|z|>1$. By using \eqref{phipsi} we see that $\Psi_n$
converge uniformly on compact subsets of the unit disk to
\begin{equation}\label{hatpsi}
\Psi(z)=\begin{bmatrix} 0\\ v(z)\end{bmatrix}, \qquad
v(z)=u(z^{-1};\beta,\alpha). \end{equation}

Thus we have proved,

\begin{theorem}\label{szego}
Suppose $\alpha_n$ and $\beta_n$ satisfy $\eqref{ltwo}$. Then
$\Phi_n(z)$ converges uniformly on compact subsets of $|z|<1$ to
$\eqref{hatphi}$, and $\Psi_n(z)$ converges uniformly on compact
subsets of $|z|>1$ to $\eqref{hatpsi}$.
\end{theorem}

\section{Tauberian results}

We now consider what happens for $|z|=1$, and also the boundary
values of the functions considered in the previous section. Although
$\{z^n\}$ is no longer in $\ell^1$, the above argument goes through
as long as we require further assumptions on the coefficients.
Similarly to \eqref{qij} we define
\begin{equation}\label{tildeqij}
\tilde q_{12}(z,n)=\sum^n_{k=1} \alpha_k z^{k-n}, \qquad \tilde
q_{21}(z,n)=-\sum^\infty_{k=n+1}\beta_k z^{n-k}.
\end{equation}
Clearly, Lemma \ref{bzlu} holds for $\tilde q_{ij}$ as well if
$|z|<1$ is replaced by $|z|>1$. Next, put
\begin{equation}\label{unitcirboundone}
M=\left\{z=e^{i\theta}:
\left|\sum_{k=1}^{\infty}\beta_{k}z^{-k}\right|+
\left|\sum_{k=1}^{\infty}\alpha_k z^k\right|<\infty
\right\}\end{equation}
and
\begin{equation}\label{unitcir}
\begin{split}
E &=\{z=e^{i\theta}:
\sum_{k=0}^{\infty}|\beta_{k+1}q_{21}(z,k)|<\infty\},\\
\tilde E &=\{z=e^{i\theta}:\sum_{k=0}^{\infty}|\alpha_{k+1}\tilde
q_{21}(z,k)|<\infty\}.\end{split}\end{equation} We know from the
theory of Fourier series that under assumption \eqref{ltwo} $M$ is
the set of full Lebesgue measure, and for $z\in M$
$$ q_{12}(z,n)-\tilde q_{21}(z,n)=z^n\sum_{k=1}^\infty
\beta_kz^{-k}, \quad \tilde
q_{12}(z,n)-q_{21}(z,n)=z^{-n}\sum_{k=1}^\infty \alpha_kz^{k}. $$

\begin{theorem}\label{united}
Suppose $\eqref{ltwo}$ holds, and $z\in E\cap M$, then for $u_n$
$\eqref{un}$
$$\lim_{n\to\infty} u_n(e^{i\theta})=u^*(e^{i\theta}) $$
exists. Likewise if $\eqref{ltwo}$ holds, and $z\in \tilde E\cap M$,
then for $v_n$ $\eqref{vn}$
$$\lim_{n\to \infty} v_n(e^{i\theta})=
v^*(e^{i\theta}) $$ exists. \end{theorem}
\begin{proof}
 For $z\in M$ we have
$$\sup_n |q_{12}(z,n)|=C(z)<\infty,
$$
so for the matrix entries $v_{ij}$ \eqref{vij}
\begin{equation}\label{vcircle}
\begin{split}
 |v_{11}(z,n)| &\le (1+C|q_{21}(z,n)|)|\beta_{n+1}q_{21}(z,n)|,\\
 |v_{12}(z,n)| &\le C|\beta_{n+1}|(|q_{21}(z,n)|+|q_{21}(z,n+1)|+
 C|q_{21}(z,n)q_{21}(z,n+1)|),\\
 |v_{21}(z,n)| &\le |\beta_{n+1}q_{21}(z,n)q_{21}(z,n+1)|,\\
 |v_{22}(z,n)| &\le
 (1+C|q_{21}(z,n)|)|\beta_{n+1}q_{21}(z,n+1)|.\end{split}\end{equation}
Also
\begin{equation*}
\begin{split}
|z\beta_{n+1}q_{21}(z,n+1)| &\le |\beta_{n+1}q_{21}(z,n)| +
|\beta_{n+1}(zq_{21}(z,n+1)-q_{21}(z,n))|\\ &=|\beta_{n+1}q_{21}(n)|
+ |\beta_{n+1}\alpha_{n+1}|.\end{split}\end{equation*}
 Thus if $z\in E\cap M$, we find
$$\sum^\infty_{n=0} \|V(z,n)\|<\infty, $$
and as in \eqref{firstbound} above
$$\sup_n \|Y(z,n)\|=C_1(z)<\infty, \qquad z\in E\cap M. $$
As in \eqref{coordin} we see that
$$|y_{2}(z,n)-y_{2}(z,m)|\le C_2\sum^n_{k=m}\|V(z,k)\|,$$
which shows that $\lim_{n\to\infty} y_{2}(z,n)=u(z;\alpha,\beta)$
exists and is finite for $z\in E\cap M$.  From the relation between
$\hat\Phi_n$ and $Y_2$ \eqref{psiytwo} and \eqref{un} we derive
$$\kappa_n^{-2}u_n(z)=(1+\rho(z,n))y_{2}(z,n)+q_{21}(z,n)y_{1}(z,n).$$
Since $y_{1}(z,n)=0(1)$, $n\to\infty$ (cf. \eqref{coordin}), and
$q_{21}(z,n)\to 0$ for $z\in E\cap M$, we find
$$u^*(z)=\lim_{n\to\infty} u_n(z)=
\prod_{j=1}^\infty (1-\alpha_j\beta_j)^{-1}\,u(z;\alpha,\beta). $$
An analogous argument shows that if $z\in \tilde E\cap M$ then
$$v^*(z)=\lim_{n\to\infty} v_n(z)=
\prod_{j=1}^\infty
(1-\alpha_j\beta_j)^{-1}\,u(z^{-1};\beta,\alpha)$$ exists and is
finite. \end{proof}

Next, set
\begin{equation}\label{taubone}
\begin{split}
E(\theta) &=\sup_{0<r\le 1}\sum_{k=0}^{\infty}|\beta_{k+1}
q_{21}(re^{i\theta},k)|, \\
\tilde E(\theta) &=\sup_{1\le
r<\infty}\sum_{k=0}^{\infty}|\alpha_{k+1}\tilde
q_{21}(re^{i\theta},k)|,\end{split}\end{equation} and
\begin{equation}\label{supboun}
N(\theta)=\sup_{0<r\le 1,n}|q_{12}(re^{i\theta},n)|,\qquad \tilde
N(\theta)=\sup_{1\le r<\infty,n}|\tilde q_{12}(re^{i\theta},n)|.
\end{equation}
We now examine what happens when $z=re^{i\theta}$ approaches the
unit circle for certain values of the argument $\theta$.

\begin{theorem}\label{Tabu}
Suppose $\eqref{ltwo}$ holds, and $z=e^{i\theta}$ is such that
$E(\theta)<\infty$ and $N(\theta)<\infty$, then for $u$ given by
$\eqref{ualbe}$ the limit
$$\lim_{r\to 1-0} u(re^{i\theta})= u^*(e^{i\theta})$$
exists. Likewise if $z=e^{i\theta}$ is such that $\tilde
E(\theta)<\infty$ and $\tilde N(\theta)<\infty$, then the limit
$$\lim_{r\to 1+0} v(re^{i\theta}) = v^*(e^{i\theta})   $$
exists. \end{theorem}
\begin{proof}
Suppose that the first assumption holds. This implies that
inequalities \eqref{vcircle} hold with $C$ replaced by $N(\theta)$.
Since $E(\theta)<\infty$ we find that
$$\sup_{0<r\le
1}\sum^\infty_{n=0}\|V(re^{i\theta},n)\|<\infty. $$ Thus
$\sup_{0<r\le 1}\|Y(re^{i\theta},n)\|<cN(\theta)$ and
$$ \sup_{0<r\le 1} |y_{2}(re^{i\theta},n)-y_{2}(re^{i\theta},m)|
\le cN(\theta)\sup_{0<r\le 1} \sum^n_{i=m} \|V(re^{i\theta},i)\|.
$$
Hence $y_{2}(re^{i\theta},n)$ converges uniformly for $0<r\le 1$ to
$u$. So $u_n(z)$ converges $u^*(z)$ uniformly for $0<r<1$. Since
$u_n$ are continuous for $0<r\le 1$ (they are polynomials) we find
that $\lim_{r\to 1-0} u(re^{i\theta})= u_*(e^{i\theta})$. An
analogous argument proves the second statement. \end{proof}

Finally, we conclude that if $z\in M$, and both $E(\theta)$ and
$N(\theta)$ are finite, then
$$ \lim_{r\to 1-0}\lim_{n\to\infty} u_n(re^{i\theta})=
\lim_{n\to\infty}\lim_{r\to 1-0}u_n(re^{i\theta}). $$ Similarly, if
$z\in M$, and both $\tilde E(\theta)$ and $\tilde N(\theta)$ are
finite, then
$$ \lim_{r\to 1+0}\lim_{n\to\infty} v_n(re^{i\theta})=
\lim_{n\to\infty}\lim_{r\to 1+0} v_n(re^{i\theta}). $$

In the setting of polynomials orthogonal with respect to a positive
measure $\mu$ on the unit circle ($\beta_n=\bar\alpha_n$,
$|\alpha_n|<1$), the tauberian problem in the Szeg\H{o} class was
studied first in \cite[Chapter 5]{Ger61}. In this case $M=\tilde M$
and
\begin{equation}\label{nikishin}
E=\tilde E=\{z=e^{i\theta}: \sum_{k=0}^{\infty}|\bar\alpha_{k+1}
\sum_{j=k+1}^\infty \alpha_j z^{j-k}|<\infty\}. \end{equation}

Condition \eqref{nikishin} appeared in \cite{Ni}, where it was shown
to imply a bound for orthogonal polynomials
$|\log|\hat\Phi_n(z)||=O(1)$ (see also \cite{Go2}). \eqref{nikishin}
is a key ingredient in \cite{Go1}, where measures on the unit circle
with slowly decaying parameters are studied and the Benzaid--Lutz
method is applied. It is shown there that the uniform convergence in
\eqref{unitcirboundone} on an arc along with \eqref{nikishin}
implies the uniform convergence of the reversed polynomials on the
same arc ($\mu$ is said to admit a uniform asymptotic
representation). Damanik \cite{Da} proved that the singular
component $\mu$ is supported on the complement of $E$.

We complete with two examples (the first one is borrowed from
\cite[Chapter V.4]{Z}).

\begin{example} It is known that the Fourier series
$$ g_\alpha(e^{i\theta})=\sum_{k=1}^\infty {e^{ick\log
k} \over k^{1/2+\alpha}}\,e^{ik\theta}, \qquad 0<\alpha<1, \quad c>0
$$
converges uniformly on $\bf T$, and $g_\alpha\in{\rm Lip}(\alpha)$.
Hence (see \cite[formula (13.26)]{Z})
$$ q_{21}(e^{i\theta},n)=O(n^{-\alpha}\log n) $$
uniformly on $\bf T$, and the measure $\mu$ with parameters
$\alpha_n=n^{-3/4-\varepsilon}e^{icn\log n}$, $\varepsilon>0$,
admits the uniform asymptotic representation on the whole circle.
The same result holds for $\mu$ with
$\alpha_n=n^{-3/4-\varepsilon}e^{in^\alpha}$, $0<\alpha<1$.
\end{example}

\begin{example} Let
$$ \alpha_n={1\over{n^\gamma}}\sum_{j=1}^m b_je^{i\lambda_j n},
\qquad b_j\in\bf {C}, \quad \lambda_k\not=\lambda_j $$ for large
enough $n\ge n_0$. We have
$$ -e^{in\theta}q_{21}(e^{i\theta},n)=\sum_{k=n+1}^\infty
\alpha_k e^{ik\theta}=\sum_{j=1}^mb_j\sum_{k=n+1}^\infty {1\over
k^\gamma}e^{ik(\lambda_j+\theta)}
$$
so the series converges uniformly inside $\bf
{T}\backslash\{\zeta_1,\ldots,\zeta_m\}$, $\zeta_k=e^{-i\lambda_k}$
with the bound $O(1/n^{\gamma})$. Hence for $\gamma>1/2$ the measure
with such parameters admits the uniform asymptotic representation on
$\bf {T}\backslash\{\zeta_1,\ldots,\zeta_m\}$. \end{example}

\end{document}